\documentclass[draft,a4,reqno,11pt]{amsart}
\usepackage{verbatim}
\usepackage{amssymb,amsmath,amsthm}
\usepackage{amsfonts}
\usepackage{array}

\newtheorem{theorem}{Theorem}[section]
\newtheorem{lemma}{Lemma}[section]
\newtheorem{corollary}{Corollary}[theorem]

\begin{document}

\title[\tiny Generating functions for zeta stars]{\small Generating functions for multiple zeta star values}

\author[\tiny Kh.~Hessami Pilehrood]{Kh.~Hessami Pilehrood}
\address{The Fields Institute for Research in Mathematical Sciences, 222 College St, Toronto, Ontario M5T 3J1 Canada}
\email{hessamik@gmail.com}

\author{T.~Hessami Pilehrood}
\address{The Fields Institute for Research in Mathematical Sciences, 222 College St, Toronto, Ontario M5T 3J1 Canada}
\email{hessamit@gmail.com}


\begin{abstract}
We study generating functions for multiple zeta star values in general form.
These generating functions provide a connection between multiple zeta star values  and multiple Euler sums, which allows us to express 
each multiple zeta star value in terms of multiple alternating Euler sums, and specifically, reduce the length of blocks of twos in the resulting  sums.

\end{abstract}

\maketitle

\section{Introduction}
Multiple Euler sums and multiple zeta  values (MZVs) have been of interest for mathematicians and physicists for more than two decades.
The systematic study of MZVs has started from works of Hoffman and Zagier in the 1990s, although some partial historical results
are dated back to the work of Euler. In this paper, we will study generating functions for multiple zeta star values in general form.
These generating functions provide a connection between multiple zeta star values (MZSVs) and multiple Euler sums, which allows us to express 
each  MZSV in terms of multiple Euler sums, and specifically, reduce the length of blocks of twos in the resulting  sums.

To begin with precise definitions,
let  ${\mathbb N}$ be  the set of positive integers, ${\mathbb N}_0={\mathbb N}\cup\{0\}$, $\overline{\mathbb N}=\{\overline{s}: \, s\in{\mathbb N}\}$,   and 
${\mathbb D}={\mathbb N}\cup\overline{{\mathbb N}}$ be the set of signed positive integers. For all $s\in{\mathbb N}$, the absolute
value function $|\cdot |$ on ${\mathbb D}$ is defined by $|s|=|\overline{s}|=s$ and the sign function is given  by ${\rm sgn}(s)=1$, ${\rm sgn}(\overline{s})=-1$.

For ${\bf s}=(s_1,\ldots, s_m)\in{\mathbb D}^m$,  we define the (alternating) Euler sums by  nested sums with strict and non-strict inequalities, 
\begin{align} \label{001}
\zeta({\bf s})=\sum_{k_1>\cdots>k_m\ge 1}\prod_{j=1}^m\frac{({\rm sgn}(s_j))^{k_j}}{k_j^{|s_j|}}  \quad\text{and}\quad
\zeta^{\star}({\bf s})=\sum_{k_1\ge\cdots\ge k_m\ge 1}\prod_{j=1}^m\frac{({\rm sgn}(s_j))^{k_j}}{k_j^{|s_j|}},  
\end{align}
respectively, 
where $s_1\ne 1$ in order for the series to converge. If ${\bf s}\in {\mathbb N}^m$, then $\zeta({\bf s})$ is called a {\it multiple zeta value} (MZV)
and $\zeta^{\star}({\bf s})$ a {\it multiple zeta star value} (MZSV). 
We assign two characteristics to each of the sums above: the length (or depth) $\ell({\bf s}):= m$ 
and the weight $|{\bf s}| := |s_1|+\cdots+|s_m|$. 
By convention, we set $\zeta(\emptyset)=\zeta^\star(\emptyset)=1$.
By $\{s\}^m$ we denote the sequence formed by repeating the symbol $\{s\}$ $m$ times.

In particular, for $s\in{\mathbb N}$ we have
\begin{equation*}
\zeta(\overline{s})=\sum_{k=1}^{\infty}\frac{(-1)^k}{k^s}=
\begin{cases}
-\log 2, & \quad\text{if}\quad s=1; \\
(2^{1-s}-1)\zeta(s), & \quad\text{if}\quad s\ge 2,
\end{cases}
\end{equation*}
where $\zeta(s)$ is a value of the Riemann zeta function. The  simplest explicit evaluations of multiple zeta values include
$$
\zeta(\{2\}^m)=\frac{\pi^{2m}}{(2m+1)!}, \qquad \zeta^\star(\{2\}^m)=-2\zeta(\overline{2m}).
$$
These formulas follow from the Laurent series expansions of two functions $\sin(\pi z)/\pi z$ and $\pi z/\sin(\pi z)$ (see \cite{HP:17} for more details).

In 2012, Zagier \cite{Za:12} found explicit formulas for $ \zeta(\{2\}^a, 3, \{2\}^b)$ and $\zeta^{\star}(\{2\}^a, 3, \{2\}^b)$, $a, b\in{\mathbb N}_0$, in terms of rational
linear combinations of products $\zeta(m)\pi^{2n}$ with $m+2n=2a+2b+3$.
These formulas played an important role in Brown's proof \cite{Br:12} of the Hoffman conjecture \cite{Hof:97} that every multiple zeta value is a ${\mathbb Q}$-linear
combination of values $\zeta(s_1, \ldots, s_r)$ for which each $s_i$ is either $2$ or $3$.

In \cite{HP:17}, we gave another proof of the theorem of Zagier via a new representation of the generating function for $\zeta^\star(\{2\}^a, 3, \{2\}^b)$
in terms of the double series
\begin{equation} \label{firstfunc}
\sum_{a,b\ge 0}\zeta^\star(\{2\}^a, 3, \{2\}^b) x^{2a}y^{2b} = 
-2\sum_{j=1}^{\infty}\frac{(-1)^j j}{(j^2-x^2)(j^2-y^2)} - 4\sum_{j=1}^{\infty}\frac{(-1)^j}{j^2-x^2}\sum_{k=1}^{j-1}\frac{k}{k^2-y^2}
\end{equation}
by using a hypergeometric identity of Andrews  
and then evaluating the right-hand side in terms of the digamma function with the help of complex integration and the residue theorem. 
Formula (\ref{firstfunc}) also implies that $\zeta^\star(\{2\}^a, 3, \{2\}^b)$
is expressible in terms of double Euler sums
\begin{equation} \label{eusum}
\zeta^\star(\{2\}^a, 3, \{2\}^b) = -2\zeta(\overline{2a+2b+3})-4\zeta(\overline{2a+2}, 2b+1),
\end{equation}
which by \cite[Theorem 7.2]{FS:98}  can be reduced to linear combinations of products of single zeta values, giving one more proof of Zagier's theorem (see \cite[Remark 2.7]{HPT:14}).
Note that formula (\ref{eusum}) as well as the similar one
$$
\zeta^\star(\{2\}^a, 1, \{2\}^b) = -2\zeta(\overline{2a+2b+1}) -4\zeta(2a+1, \overline{2b}), \qquad a, b\in{\mathbb N},
$$
were proved  even earlier in \cite{HPT:14} by another method using finite binomial identities.  These formulas were generalized by Zhao \cite{Z:16}  and  
Linebarger and Zhao \cite{LZ:15} that led to the proof of the Two-one formula conjectured by Ohno and Zudilin \cite{OZ:08}, which states that
$$
\zeta^\star(\{2\}^{a_1}, 1, \{2\}^{a_2}, \ldots, 1, \{2\}^{a_d}, 1) = \sum_{{\bf p} = (2a_1+1)\circ\cdots\circ(2a_d+1)} 2^{\ell({\bf p})}\zeta({\bf p}),
$$
where ${\bf p}$ runs through all indices of the form $(2a_1+1)\circ\cdots\circ(2a_d+1)$ with ``$\circ$'' being either the symbol ``,'' or the sign  ``$+$''.
These results were  extended  further by the present authors and Zhao \cite[Theorem 1.4]{HPZ:16}, to get formulas for arbitrary multiple zeta star values
$\zeta^\star(\{2\}^{a_0}, c_1, \{2\}^{a_1}, \ldots, c_d, \{2\}^{a_d})$ in terms of multiple Euler sums.

In this paper, we generalize the formula for generating function (\ref{firstfunc}) to include generating functions of multiple zeta star values 
with an arbitrary number of blocks of twos. 
\begin{theorem} \label{2-3}
For any integer $d\ge 1$ and any complex numbers $z_0, z_1,\ldots, z_d$ with $|z_j|<1, j=0, 1,\ldots, d$,
we have
\begin{equation*}
\begin{split}
\sum_{a_0,  \ldots, a_d\ge 0}&\zeta^{\star}(\{2\}^{a_0}, 3, \{2\}^{a_1},  \ldots, 3, \{2\}^{a_d}) \,z_0^{2a_0}\cdots z_d^{2a_d} \\
 &=-2\sum_{k_0\ge\cdots\ge k_d\ge 1} \frac{(-1)^{k_0} k_d^2}{k_0^2-z_0^2}\, \prod_{i=1}^d
\frac{2^{\Delta(k_{i-1}, k_i)}}{k_i(k_i^2-z_i^2)},
\end{split}
\end{equation*}
where
$$
\Delta(a,b)=\begin{cases}
0, &\quad\text{if}\quad a=b; \\
1, &\quad\text{else}.
\end{cases}
$$
In particular,
\begin{equation*}
\sum_{a_0\ge 0}\zeta^\star(\{2\}^{a_0})\,z_0^{2a_0}=1-2z_0^2\,\sum_{k_0\ge 1}\frac{(-1)^{k_0}}{k_0^2-z_0^2}.
\end{equation*}
\end{theorem}
In the same vein we give the  formulas for  the Two-one and Two-three-two-one generating functions.
\begin{theorem} \label{2-1-2}
For any integer $d\ge 0$ and any complex numbers $z_0, z_1,\ldots, z_d$ with $|z_j|<1, j=0, 1,\ldots, d$,
we have
\begin{equation*}
\begin{split}
\sum_{a_0,   \ldots, a_d\ge 0}&\zeta^{\star}(\{2\}^{a_0+1}, 1, \{2\}^{a_1},  \ldots, 1, \{2\}^{a_d}) \,z_0^{2a_0}\cdots z_d^{2a_d} \\
& =-\sum_{k_0\ge \cdots\ge k_d\ge 1} \frac{(-1)^{k_d} k_d}{k_0^2}\, \prod_{i=0}^d
\frac{k_i \cdot 2^{\Delta(k_{i-1}, k_i)}}{k_i^2-z_i^2}
\end{split}
\end{equation*}
and
\begin{equation*}
\begin{split}
\sum_{a_0,   \ldots, a_{d}\ge 0}&\zeta^{\star}(\{2\}^{a_0+1}, 1, \{2\}^{a_1},  \ldots, 1, \{2\}^{a_{d}}, 1) \,z_0^{2a_0}\cdots z_{d}^{2a_{d}} \\
& = \sum_{k_0\ge\cdots\ge k_{d}\ge 1} \frac{1}{k_0^2}\, \prod_{i=0}^{d}
\frac{k_i \cdot 2^{\Delta(k_{i-1}, k_i)}}{k_i^2-z_i^2},
\end{split}
\end{equation*}
where $k_{-1}=0$.
\end{theorem}
\begin{theorem} \label{2-3-1}
For any integer $d\ge 1$ and any complex numbers $z_0, z_1,\ldots, z_{2d}$ with $|z_j|<1, j=0, 1,\ldots, 2d$,
we have
\begin{equation*}
\begin{split}
\sum_{a_0, \ldots, a_{2d}\ge 0}&\zeta^{\star}(\{2\}^{a_0}, 3, \{2\}^{a_1}, 1,  \ldots, 3, \{2\}^{a_{2d-1}}, 1, \{2\}^{a_{2d}}) \,z_0^{2a_0}\cdots z_{2d}^{2a_{2d}} \\
& =-\sum_{k_0\ge\cdots\ge k_{2d}\ge 1} k_{2d}^2\, \prod_{i=0}^{2d}
\frac{(-1)^{k_i} \cdot 2^{\Delta(k_{i-1}, k_i)}}{k_i^2-z_i^2},
\end{split}
\end{equation*}
where $k_{-1}=0$, and
\begin{equation*}
\begin{split}
\sum_{a_1,  \ldots, a_{2d}\ge 0}&\zeta^{\star}(\{2\}^{a_1}, 3, \{2\}^{a_2}, 1,  \ldots, 3, \{2\}^{a_{2d}}, 1) \,z_1^{2a_1}\cdots z_{2d}^{2a_{2d}} \\
& =\sum_{k_1\ge\cdots\ge k_{2d}\ge 1} \,\, \prod_{i=1}^{2d}
\frac{(-1)^{k_i} \cdot 2^{\Delta(k_{i-1}, k_i)}}{k_i^2-z_i^2},
\end{split}
\end{equation*}
where $k_0=0$.
\end{theorem}
For a string of positive integers ${\mathbf r}=(r_1, \ldots, r_c) \in{\mathbb N}^c$ and positive integers $k, m$, define the multiple sharp sum
$$
S^{\sharp}_{k,m}(r_1, \ldots, r_c)=\begin{cases}
\sum\limits_{k\ge l_1\ge \cdots\ge l_c\ge m}\frac{2^{\Delta(k, l_1)+\Delta(l_1, l_2)+\cdots+\Delta(l_{c-1}, l_c)+\Delta(l_c,m)}}{l_1^{r_1}l_2^{r_2}\cdots l_c^{r_c}},
&\quad\text{if}\quad {\mathbf r}\ne\emptyset\,\,\,\,\text{and}\,\,\, k\ge m; \\[10pt]
\qquad 2^{\Delta(k,m)}, &\quad\text{if} \quad {\mathbf r}=\emptyset \,\,\,\,\,\,\text{or}\,\,\,\, \,k<m.
\end{cases}
$$
Actually, our main result is more general than the above three theorems, since it provides formulas for generating functions of arbitrary multiple zeta star values
with an arbitrary number of blocks of twos. 
\begin{theorem} \label{T2}
For any integers  $d\ge 1$, $c_1, \ldots, c_d\in{\mathbb N}\setminus{\{2\}}$, $c_1\ge 3$, and any complex numbers $z_0, z_1, \ldots, z_d$ with $|z_j|<1, j=0, 1,\ldots, d$,
we have
\begin{equation} \label{mainth} 
\begin{split}
\sum_{a_0, a_1, \ldots, a_d\ge 0}&\zeta^{\star}(\{2\}^{a_0}, c_1, \{2\}^{a_1},  \ldots, c_d, \{2\}^{a_d}) \,z_0^{2a_0}z_1^{2a_1}\cdots z_d^{2a_d} \\
&=-\sum_{k_0\ge k_1\ge\cdots\ge k_d\ge 1} \prod_{i=0}^d
\frac{(-1)^{k_i\delta_i}k_i^{\delta_i-1}}{k_i^2-z_i^2}\,S^{\sharp}_{k_{i-1}, k_i}(\{1\}^{c_i-3}),
\end{split}
\end{equation}
where $c_0=1, c_{d+1}=0, k_{-1}=0$, $\delta_i=\delta(c_i)+\delta(c_{i+1})$, and
$\delta(c)=\begin{cases}
2, &\quad\text{if}\,\, c=0; \\
1, &\quad\text{if}\,\, c=1; \\
0, &\quad\text{if}\,\, c\ge 3.
\end{cases}
$
\end{theorem}

\noindent If all $c_i$ take only values 1 or 3, the formula can be simplified.

\begin{corollary} \label{CZ1}
For any  integers $d\ge 1$, $c_1, \ldots, c_d\in \!\{1, 3\}$, $c_1\ge 3$, and any complex numbers $z_0, z_1,\ldots, z_d$ with $|z_j|<1, j=0, 1,\ldots, d$,
we have
\begin{equation*}
\begin{split}
\sum_{a_0, a_1, \ldots, a_d\ge 0}&\zeta^{\star}(\{2\}^{a_0}, c_1, \{2\}^{a_1},  \ldots, c_d, \{2\}^{a_d}) \,z_0^{2a_0}z_1^{2a_1}\cdots z_d^{2a_d} \\
&=-\sum_{k_0\ge k_1\ge\cdots\ge k_d\ge 1}\prod_{i=0}^d
\frac{(-1)^{k_i\delta_i}k_i^{\delta_i-1}}{k_i^2-z_i^2}\,2^{\Delta(k_{i-1}, k_i)}
\end{split}
\end{equation*}
with the same notation as in Theorem {\rm\ref{T2}}.
\end{corollary}

\noindent  The case $z_d=0$ and $c_d=1$ for $d\ge 2$  also leads to the simplification of the right hand side of~(\ref{mainth}).  Note that after
substitution $z_d=0$ and $c_d=1$ in (\ref{mainth}), we replace $d-1$ by $d$ to get the next corollary.

\begin{corollary} \label{CZ3}
For any integers  $d\ge 1$,  $c_1, \ldots, c_{d}\in{\mathbb N}\setminus{\{2\}}$, $c_1\ge 3$,  and any complex numbers $z_0,  \ldots, z_{d}$ 
with $|z_j|<1, j=0, 1,\ldots, d$,
we have
\begin{equation*} 
\begin{split}
\sum_{a_0, a_1, \ldots, a_{d}\ge 0}&\zeta^{\star}(\{2\}^{a_0}, c_1, \{2\}^{a_1},  \ldots, c_{d}, \{2\}^{a_{d}}, 1) \,z_0^{2a_0}z_1^{2a_1}\cdots z_{d}^{2a_{d}} \\
&=\sum_{k_0\ge k_1\ge\cdots\ge k_{d}\ge 1} \prod_{i=0}^{d}
\frac{(-1)^{k_i\delta_i}k_i^{\delta_i-1}}{k_i^2-z_i^2}\,S^{\sharp}_{k_{i-1}, k_i}(\{1\}^{c_i-3}),
\end{split}
\end{equation*}
where $c_0=c_{d+1}=1$, $k_{-1}=0$, $\delta_i=\delta(c_i)+\delta(c_{i+1})$, and $\delta(c)$
is defined in Theorem {\rm \ref{T2}}.
\end{corollary}

\begin{theorem} \label{CZ2}
For any integers $d\ge 0$, $0\le m\le d$, $c_1, \ldots, c_d\in{\mathbb N}\setminus{\{2\}}$ such that $c_1\ge 3$ if $m\ge 1$, and any complex numbers $z_0, z_1, \ldots, z_d$ 
with $|z_j|<1, j=0, 1,\ldots, d$,
we have
\begin{equation} \label{4.5}
\begin{split}
\underset{a_m\ge 1}{\sum_{a_0, a_1, \ldots, a_d\ge 0}}&\zeta^{\star}(\{2\}^{a_0}, c_1, \{2\}^{a_1},  \ldots, c_d, \{2\}^{a_d}) \,z_0^{2a_0}z_1^{2a_1}\cdots z_d^{2a_d} \\
&=-z_m^2\sum_{k_0\ge k_1\ge\cdots\ge k_d\ge 1}\frac{1}{k_m^2}\,\prod_{i=0}^d
\frac{(-1)^{k_i\delta_i}k_i^{\delta_i-1}}{k_i^2-z_i^2}\,S^{\sharp}_{k_{i-1}, k_i}(\{1\}^{c_i-3})
\end{split}
\end{equation}
with the same notation as in Theorem {\rm \ref{T2}}. 
\end{theorem}

\begin{corollary} \label{CZ3.5}
For any integers $d\ge 0$,  $0\le m\le d$, $c_1, \ldots, c_{d}\in{\mathbb N}\setminus{\{2\}}$ such that $c_1\ge 3$ if $m\ge 1$, and any complex numbers $z_0,  \ldots, z_{d}$ 
with $|z_j|<1, j=0, 1,\ldots, d$,
we have
\begin{equation*} 
\begin{split}
\underset{a_m\ge 1}{\sum_{a_0, a_1, \ldots, a_{d}\ge 0}}&\zeta^{\star}(\{2\}^{a_0}, c_1, \{2\}^{a_1},  \ldots, c_{d}, \{2\}^{a_{d}}, 1) \,z_0^{2a_0}z_1^{2a_1}\cdots z_{d}^{2a_{d}} \\
&=z_m^2\sum_{k_0\ge k_1\ge\cdots\ge k_{d}\ge 1}\,\frac{1}{k_m^2}\, \prod_{i=0}^{d}
\frac{(-1)^{k_i\delta_i}k_i^{\delta_i-1}}{k_i^2-z_i^2}\,S^{\sharp}_{k_{i-1}, k_i}(\{1\}^{c_i-3})
\end{split}
\end{equation*}
where $c_0=c_{d+1}=1$, $k_{-1}=0$, $\delta_i=\delta(c_i)+\delta(c_{i+1})$, and $\delta(c)$
is defined in Theorem {\rm \ref{T2}}.
\end{corollary}
The next theorem provides a formula for arbitrary MZSV in terms of Euler sums. One can consider this type of relations as duality relations
between multiple zeta star values and multiple Euler sharp sums. For different types of duality relations known so far for multiple Euler sums, see \cite[Section 6]{BL}.

\begin{theorem} \label{CZ4}
For any integers $d\ge 0$,  $c_1, \ldots, c_{d}\in{\mathbb N}\setminus{\{2\}}$,  $a_0, a_1  \ldots, a_{d}\ge 0$, and $c_1\ge 3$ if $a_0=0$ and $d\ge 1$, and $a_0\ge 1$ if $d=0$,
we have
\begin{equation} \label{CZ4eq1} 
\zeta^{\star}(\{2\}^{a_0}, c_1, \{2\}^{a_1},  \ldots, c_{d}, \{2\}^{a_{d}})
=-\sum_{k_0\ge k_1\ge\cdots\ge k_{d}\ge 1}\,\prod_{i=0}^{d}\,
\frac{(-1)^{k_i\delta_i}}{k_i^{2a_i+3-\delta_i}}\,S^{\sharp}_{k_{i-1}, k_i}(\{1\}^{c_i-3})
\end{equation}
with the same notation as in Theorem {\rm \ref{T2}}, and
\begin{equation} \label{CZ4eq2} 
\zeta^{\star}(\{2\}^{a_0}, c_1, \{2\}^{a_1},  \ldots, c_{d}, \{2\}^{a_{d}}, 1)
=\sum_{k_0\ge\cdots\ge k_{d}\ge 1}\, \prod_{i=0}^{d}\,
\frac{(-1)^{k_i\delta_i}}{k_i^{2a_i+3-\delta_i}}\,S^{\sharp}_{k_{i-1}, k_i}(\{1\}^{c_i-3}),
\end{equation}
where $c_0=c_{d+1}=1$, $k_{-1}=0$, $\delta_i=\delta(c_i)+\delta(c_{i+1})$, and $\delta(c)$
is defined in Theorem {\rm \ref{T2}}.
Moreover, if $c_i\in\{1, 3\}$, then $S^{\sharp}_{k_{i-1}, k_i}(\{1\}^{c_i-3})$ is replaced by $2^{\Delta(k_{i-1}, k_i)}$ in the above formulas.
\end{theorem}
Note that Theorem \ref{CZ4} implies  \cite[Theorem 1.4]{HPZ:16} if we expand and regroup the inner sharp sums 
in powers of $2$. 

Summarizing, it is worth mentioning  that generating functions (\ref{mainth})  and (\ref{4.5}) in their generality and simplicity may be very useful in applications. For example, in  \cite{HP:18},
we show how to apply generating functions obtained in this paper for evaluation of certain explicit formulas as well as  sum formulas for multiple zeta star values on 3-2-1 indices.

\section{Auxiliary Statements}

In this section, we prove several lemmas that will be needed in the sequel. From \cite[(2.1), (2.2), and (2.5)]{HPT:14} we have the following statement.
\begin{lemma}  \label{L1}
For any positive integer $n$ and a non-negative integer $l$, we have
\begin{align}
2\sum_{k=l+1}^n (-1)^k \frac{\binom{n}{k}}{\binom{n+k}{k}} &= \frac{(-1)^{l+1}}{n}\cdot\frac{(n-l)\binom{n}{l}}{\binom{n+l}{l}}, \label{part1}\\
2\sum_{k=l+1}^n\frac{k\binom{n}{k}}{\binom{n+k}{k}} &= \frac{(n-l)\binom{n}{l}}{\binom{n+l}{l}}. \label{part2}
\end{align}
If $n\ge 2$, then
\begin{equation}
\sum_{k=1}^n\frac{(-1)^kk^2\binom{n}{k}}{\binom{n+k}{k}}=0. \label{part3}
\end{equation}
\end{lemma}
\begin{lemma} \label{L2}
For any positive integers $n, l$ and a non-negative integer $c$, we have
\begin{align}
\sum_{k=l}^n\frac{k\binom{n}{k}\,2^{\Delta(k,l)}}{\binom{n+k}{k}} &= \frac{n\binom{n}{l}}{\binom{n+l}{l}}, \label{Part1} \\
\sum_{k=l}^n\frac{(-1)^k\binom{n}{k}}{\binom{n+k}{k}}\,S^{\sharp}_{k,l}(\{1\}^c) &= \frac{(-1)^l}{n^{c+1}}\cdot\frac{l\binom{n}{l}}{\binom{n+l}{l}}. \label{Part2}
\end{align}
\end{lemma} 
\begin{proof} The proof of (\ref{Part1}) easily follows from Lemma \ref{L1}, identity (\ref{part2}),
\begin{equation*}
\sum_{k=l}^n\frac{k\binom{n}{k}\,2^{\Delta(k,l)}}{\binom{n+k}{k}}=\frac{l\binom{n}{l}}{\binom{n+l}{l}}+2\sum_{k=l+1}^n\frac{k\binom{n}{k}}{\binom{n+k}{k}}
=\frac{l\binom{n}{l}}{\binom{n+l}{l}}+\frac{(n-l)\binom{n}{l}}{\binom{n+l}{l}}=\frac{n\binom{n}{l}}{\binom{n+l}{l}}.
\end{equation*}
To prove (\ref{Part2}), we apply induction on $c$. For $c=0$, $S^{\sharp}_{k,l}(\{1\}^c)=2^{\Delta(k,l)}$ and we have by (\ref{part1})
\begin{equation} \label{eq03}
\sum_{k=l}^n\frac{(-1)^k\binom{n}{k}\,2^{\Delta(k,l)}}{\binom{n+k}{k}} = \frac{(-1)^l\binom{n}{l}}{\binom{n+l}{l}} + 
2\sum_{k=l+1}^n\frac{(-1)^k\binom{n}{k}}{\binom{n+k}{k}}=
\frac{(-1)^l}{n}\cdot\frac{l\binom{n}{l}}{\binom{n+l}{l}}.
\end{equation}
If $c\ge 1$, then changing the order of summation and applying identity (\ref{eq03}), we obtain
\begin{equation*}
\begin{split}
\sum_{k=l}^n\frac{(-1)^k\binom{n}{k}}{\binom{n+k}{k}}\,S^{\sharp}_{k,l}(\{1\}^c) &=
\sum_{k=l}^n\frac{(-1)^k\binom{n}{k}}{\binom{n+k}{k}}\sum_{k\ge l_1\ge \cdots\ge l_c\ge l}\frac{2^{\Delta(k,l_1)+\Delta(l_1,l_2)+\cdots+\Delta(l_c,l)}}{l_1l_2\cdots l_c} \\
&=\sum_{n\ge l_1\ge \cdots l_c\ge l} \left(\sum_{k=l_1}^n\frac{(-1)^k\binom{n}{k}\,2^{\Delta(k,l_1)}}{\binom{n+k}{k}}\right) \frac{2^{\Delta(l_1,l_2)+
\cdots+\Delta(l_c,l)}}{l_1l_2\cdots l_c} \\
&=\sum_{n\ge l_1\ge \cdots l_c\ge l}\frac{(-1)^{l_1}}{n}\cdot\frac{l_1\binom{n}{l_1}}{\binom{n+l_1}{l_1}}
\cdot\frac{2^{\Delta(l_1,l_2)+\cdots+\Delta(l_c,l)}}{l_1l_2\cdots l_c} \\
&=\frac{1}{n}\sum_{l_1=l}^n\frac{(-1)^{l_1}\binom{n}{l_1}}{\binom{n+l_1}{l_1}}\, S^{\sharp}_{l_1,l}(\{1\}^{c-1}).
\end{split}
\end{equation*}
Now formula (\ref{Part2}) easily follows by induction on $c$.
\end{proof}
The proof of the next lemma was essentially given in detail in \cite[Lemma 4.2]{LZ:15}. We slightly modified the formulation  to embrace 
a more general class of series.
\begin{lemma}   \label{L3}
Let $M, c, a\in{\mathbb R}, M>0, a>1$, and let $R_k$ be a sequence of real numbers satisfying 
$|R_k|<\frac{M(\log k+1)^c}{k^a}$ for all $k=1,2,\ldots$. Then
$$
\lim_{n\to\infty}\sum_{k=1}^n|R_k|\left(1-\frac{\binom{n}{k}}{\binom{n+k}{k}}\right)=0.
$$
\end{lemma}

\section{Generating Functions For Multiple Harmonic Sums}

In this section, we prove  a finite version of identity (\ref{mainth}), from which Theorem \ref{T2} will follow by limit transition.
For any $n, m \in{\mathbb N}$ and
${\bf s}=(s_1,\ldots, s_m)\in{\mathbb D}^m$,  we  define the (alternating) multiple harmonic sums by 
\begin{align*} 
H_n({\bf s})=\sum_{n\ge k_1>\cdots>k_m\ge 1}\prod_{j=1}^m\frac{({\rm sgn}(s_j))^{k_j}}{k_j^{|s_j|}}  \quad\text{and}\quad
H_n^{\star}({\bf s})=\sum_{n\ge k_1\ge\cdots\ge k_m\ge 1}\prod_{j=1}^m\frac{({\rm sgn}(s_j))^{k_j}}{k_j^{|s_j|}}.
\end{align*}
By convention, we put $H_n({\bf s})=0$ if $n<m$, and $H_n(\emptyset)=H_n^\star(\emptyset)=1$.

\noindent Let $F_n(c_1, \ldots, c_d; z_0, z_1, \ldots, z_d)$ denote the generating function of the multiple harmonic star sum,
\begin{equation*} 
F_n(c_1, \ldots, c_d; z_0, z_1, \ldots, z_d)=
\sum_{a_0, a_1, \ldots, a_d\ge 0}H_n^{\star}(\{2\}^{a_0}, c_1, \{2\}^{a_1},  \ldots, c_d, \{2\}^{a_d}) \,z_0^{2a_0}z_1^{2a_1}\cdots z_d^{2a_d}.
\end{equation*}
 
\begin{theorem} \label{T1}
For any integers $n\ge 1$, $d\ge 0$, $c_1, \ldots, c_d\in{\mathbb N}\setminus{\{2\}}$, and any complex numbers $z_0, z_1, \ldots, z_d$ with $|z_j|<1, j=0, 1,\ldots, d$,
we have
\begin{equation} \label{T1part1}
\begin{split}
&F_n(c_1, \ldots, c_d; z_0, z_1, \ldots, z_d) \\
&\qquad\qquad\quad=-\sum_{n\ge k_0\ge k_1\ge\cdots\ge k_d\ge 1}\frac{\binom{n}{k_0}}{\binom{n+k_0}{k_0}} \prod_{i=0}^d
\frac{(-1)^{k_i\delta_i}k_i^{\delta_i-1}}{k_i^2-z_i^2}\,S^{\sharp}_{k_{i-1}, k_i}(\{1\}^{c_i-3}),
\end{split}
\end{equation}
where $c_0=1, c_{d+1}=0, k_{-1}=0$, $\delta_i=\delta(c_i)+\delta(c_{i+1})$, and
$\delta(c)=\begin{cases}
2, &\quad\text{if}\,\, c=0; \\
1, &\quad\text{if}\,\, c=1; \\
0, &\quad\text{if}\,\, c\ge 3.
\end{cases}
$
\end{theorem}
\begin{proof} If $n=1$, the theorem is obviously true. We have
\begin{equation*}
F_1(c_1, \ldots, c_d; z_0, z_1, \ldots, z_d)
=\sum_{a_0, a_1, \ldots, a_d\ge 0}z_0^{2a_0}z_1^{2a_1}\cdots z_d^{2a_d}=\prod_{j=0}^d\frac{1}{1-z_j^2},
\end{equation*}
and the right-hand side of (\ref{T1part1}) is
\begin{equation*}
-\frac{\binom{1}{1}}{\binom{2}{1}}\left(\prod_{i=0}^d\frac{(-1)^{\delta_i}}{1-z_i^2}\right) 2^{\Delta(0,1)}=(-1)^{1+\delta_0+\delta_1+\cdots+\delta_d}\prod_{i=0}^d
\frac{1}{1-z_i^2}=\prod_{i=0}^d
\frac{1}{1-z_i^2}.
\end{equation*}
If $d=0$, the formula becomes
\begin{equation}
F_n(\, ; z_0)=\sum_{a_0\ge 0}H_n^\star(\{2\}^{a_0})\,z_0^{2a_0}=-2\sum_{n\ge k_0\ge 1}\frac{\binom{n}{k_0}}{\binom{n+k_0}{k_0}}\cdot\frac{(-1)^{k_0}\,k_0^2}{k_0^2-z_0^2}.
\label{d0}
\end{equation}
Then it is easy to see that
\begin{equation*}
\begin{split}
F_n(\, ; z_0)&=\sum_{a_0\ge 0}H_n^\star(\{2\}^{a_0})\,z_0^{2a_0}=\sum_{a_0=0}^{\infty}z_0^{2a_0}\sum_{k=0}^{a_0}\frac{1}{n^{2(a_0-k)}}H_{n-1}^\star(\{2\}^k)\\
&=\sum_{k=0}^{\infty}H_{n-1}^\star(\{2\}^k)z_0^{2k}\sum_{a_0=k}^{\infty}\frac{z^{2(a_0-k)}}{n^{2(a_0-k)}}=\frac{n^2}{n^2-z^2}F_{n-1}(\, ; z_0)
\end{split}
\end{equation*}
and (\ref{d0}) follows immediately by induction on $n$. Indeed, we have
\begin{equation*}
\begin{split}
F_n(\, ; z_0)&=\frac{n^2}{n^2-z_0^2}\,F_{n-1}(\, ; z_0)=\frac{-2n^2}{n^2-z_0^2}\sum_{ k=1}^{n-1}\frac{\binom{n-1}{k}}{\binom{n-1+k}{k}}\frac{(-1)^kk^2}{k^2-z_0^2} \\
&=2\sum_{k=1}^n\frac{(-1)^k k^2}{k^2-z_0^2}\frac{\binom{n}{k}}{\binom{n+k}{k}}\!\left(\frac{k^2-z_0^2}{n^2-z_0^2}-1\right) 
=-2\sum_{k=1}^n\frac{(-1)^kk^2}{k^2-z_0^2}\frac{\binom{n}{k}}{\binom{n+k}{k}},
\end{split}
\end{equation*}
where in the last equality we used identity (\ref{part3}).

For the general case $n>1$ and $d>0$, we proceed
by induction on $n+d$. When $n+d=1$ or 2 the formula is true by the above.
Let $N$ be a positive integer. Suppose (\ref{T1part1}) is true for all $n+d\le N$.  To prove it for $n+d=N+1$ with $n>1$ and $d>0$, we consider
the expansion of the finite sum
\begin{equation*}
\begin{split}
H_n^\star(\{2\}^{a_0}, c_1, \{2\}^{a_1}, \ldots, c_d, \{2\}^{a_d})&=\frac{1}{n^{2a_0+c_1}}\,H_n^\star(\{2\}^{a_1}, c_2, \{2\}^{a_2}, \ldots, c_d, \{2\}^{a_d}) \\
&+\sum_{k=0}^{a_0}\frac{1}{n^{2(a_0-k)}}\,H_{n-1}^\star(\{2\}^k, c_1, \{2\}^{a_1}, \ldots, c_d, \{2\}^{a_d}),
\end{split}
\end{equation*}
which leads to the following reduction of the generating function:
\begin{equation*}
\begin{split}
F_n(c_1, \ldots, c_d; &\, z_0, z_1, \ldots, z_d)=\sum_{a_0\ge 0}\frac{z_0^{2a_0}}{n^{2a_0+c_1}}\cdot F_n(c_2, \ldots, c_d; z_1, \ldots, z_d) \\
&+\sum_{a_0, a_1, \ldots, a_d\ge 0}\sum_{k=0}^{a_0}\frac{1}{n^{2(a_0-k)}}\,H_{n-1}^\star(\{2\}^k, c_1, \{2\}^{a_1}, \ldots, c_d, \{2\}^{a_d})\,z_0^{2a_0}z_1^{2a_1}
\cdots z_d^{2a_d}.
\end{split}
\end{equation*}
Changing the order of summation in the second sum and simplifying, we get
\begin{equation*}
\begin{split}
\sum_{a_0, a_1, \ldots, a_d\ge 0}&\sum_{k=0}^{a_0}\frac{1}{n^{2(a_0-k)}}\,H_{n-1}^\star(\{2\}^k, c_1, \{2\}^{a_1}, \ldots, c_d, \{2\}^{a_d})\,z_0^{2a_0}z_1^{2a_1}
\cdots z_d^{2a_d} \\
&=\sum_{k=0}^{\infty}\sum_{a_1,\ldots, a_d\ge 0}H_{n-1}^\star(\{2\}^k, c_1, \{2\}^{a_1}, \ldots, c_d, \{2\}^{a_d})z_0^{2k}z_1^{2a_1}\cdots z_d^{2a_d}\cdot
\sum_{a_0=k}^{\infty}\frac{z_0^{2(a_0-k)}}{n^{2(a_0-k)}} \\
&=\frac{n^2}{n^2-z_0^2}\cdot F_{n-1}(c_1, \ldots, c_d;  z_0, z_1, \ldots, z_d)
\end{split}
\end{equation*}
and therefore,
\begin{equation} \label{rec}
\begin{split}
F_n(c_1, \ldots, c_d; z_0, z_1, \ldots, z_d)&=\frac{n^{2-c_1}}{n^2-z_0^2}\cdot F_n(c_2, \ldots, c_d; z_1, \ldots, z_d) \\
&+ \frac{n^2}{n^2-z_0^2}\cdot F_{n-1}(c_1, \ldots, c_d; z_0, z_1, \ldots, z_d).
\end{split}
\end{equation}
Let $\Sigma_0$ denote the right-hand side of (\ref{T1part1}). Consider the difference $F_n(c_1, \ldots, c_d; z_0, z_1, \ldots, z_d)-\Sigma_0$
and evaluate $\frac{n^2}{n^2-z_0^2}\,F_{n-1}(c_1, \ldots, c_d; z_0, z_1, \ldots, z_d)-\Sigma_0$. 
By the induction hypothesis for $(n-1)+d=n+(d-1)=N$, we have
\begin{equation*}
\begin{split}
\frac{n^2}{n^2-z_0^2}\,&F_{n-1}(c_1, \ldots, c_d; z_0, z_1, \ldots, z_d)-\Sigma_0 \\
&= \frac{-n^2}{n^2-z_0^2}\sum_{n-1\ge k_0\ge k_1\ge \cdots\ge k_d\ge 1}
\frac{\binom{n-1}{k_0}}{\binom{n-1+k_0}{k_0}}\,\prod_{i=0}^d \frac{(-1)^{k_i\delta_i}k_i^{\delta_i-1}}{k_i^2-z_i^2} S^{\sharp}_{k_{i-1}, k_i}(\{1\}^{c_i-3}) \\
&\,+ \sum_{n\ge k_0\ge k_1\ge\cdots\ge k_d\ge 1}\frac{\binom{n}{k_0}}{\binom{n+k_0}{k_0}} \prod_{i=0}^d
\frac{(-1)^{k_i\delta_i}k_i^{\delta_i-1}}{k_i^2-z_i^2}\,S^{\sharp}_{k_{i-1}, k_i}(\{1\}^{c_i-3}) \\
&=\sum_{n\ge k_0\ge k_1\ge\cdots\ge k_d\ge 1}\prod_{i=0}^d
\frac{(-1)^{k_i\delta_i}k_i^{\delta_i-1}}{k_i^2-z_i^2}\,S^{\sharp}_{k_{i-1}, k_i}(\{1\}^{c_i-3})
\left(\frac{\binom{n}{k_0}}{\binom{n+k_0}{k_0}} - \frac{n^2}{n^2-z_0^2}\cdot\frac{\binom{n-1}{k_0}}{\binom{n-1+k_0}{k_0}}\right). 
\end{split}
\end{equation*}
Notice that the expression in the parenthesis simplifies to
$$
\frac{\binom{n}{k_0}}{\binom{n+k_0}{k_0}} - \frac{n^2}{n^2-z_0^2}\cdot\frac{\binom{n-1}{k_0}}{\binom{n-1+k_0}{k_0}}
=\frac{\binom{n}{k_0}}{\binom{n+k_0}{k_0}}\cdot \frac{k_0^2-z_0^2}{n^2-z_0^2}
$$
and therefore,
\begin{equation} \label{eq04}
\begin{split}
&\frac{n^2}{n^2-z_0^2}\,F_{n-1}(c_1, \ldots, c_d; z_0, z_1, \ldots, z_d)-\Sigma_0 \\
&\quad= \frac{2}{n^2-z_0^2}\sum_{n\ge k_0\ge k_1\ge\cdots\ge k_d\ge 1}\frac{\binom{n}{k_0}}{\binom{n+k_0}{k_0}}
(-1)^{k_0(1+\delta(c_1))} k_0^{\delta(c_1)} \prod_{i=1}^d\frac{(-1)^{k_i\delta_i} k_i^{\delta_i-1}}{k_i^2-z_i^2}\, S^{\sharp}_{k_{i-1}, k_i}(\{1\}^{c_i-3}).
\end{split}
\end{equation}
Now let us evaluate the inner single sum over $k_0$ in (\ref{eq04}), which is
$$
\Sigma_{k_0}:=\sum_{k_0=k_1}^n\frac{\binom{n}{k_0}}{\binom{n+k_0}{k_0}}\,(-1)^{k_0(1+\delta(c_1))}\cdot k_0^{\delta(c_1)}\cdot S^{\sharp}_{k_0, k_1}(\{1\}^{c_1-3}).
$$
If $c_1=1$, then $\delta(c_1)=1$ and we get by Lemma \ref{L2}, (\ref{Part1}),
$$
\Sigma_{k_0}=\sum_{k_0=k_1}^n\frac{\binom{n}{k_0}}{\binom{n+k_0}{k_0}}\cdot k_0\cdot 2^{\Delta(k_0,k_1)} = \frac{n\binom{n}{k_1}}{\binom{n+k_1}{k_1}},
$$
which implies
\begin{equation*} 
\begin{split}
&\frac{n^2}{n^2-z_0^2}\,F_{n-1}(c_1, \ldots, c_d; z_0, z_1, \ldots, z_d)-\Sigma_0 \\
&\quad= \frac{2n}{n^2-z_0^2}\sum_{n\ge k_1\ge\cdots\ge k_d\ge 1}\frac{\binom{n}{k_1}}{\binom{n+k_1}{k_1}}
\frac{(-1)^{k_1(1+\delta(c_2))} k_1^{\delta(c_2)}}{k_1^2-z_1^2} \prod_{i=2}^d\frac{(-1)^{k_i\delta_i} k_i^{\delta_i-1}}{k_i^2-z_i^2}\, S^{\sharp}_{k_{i-1}, k_i}(\{1\}^{c_i-3}) \\
&\quad=\frac{-n}{n^2-z_0^2}\cdot F_n(c_2, \ldots, c_d; z_1, \ldots, z_d),
\end{split}
\end{equation*}
where in the last equality we applied the induction hypothesis for $n+(d-1)=N$. Now from the above and recurrence (\ref{rec}) we conclude that
$F_n(c_1, \ldots, c_d; z_0, z_1, \ldots, z_d)=\Sigma_0$ and therefore,  the theorem is proved  in this case.

If $c_1\ge 3$, then $\delta(c_1)=0$ and by Lemma \ref{L2}, (\ref{Part2}), we obtain
$$
\Sigma_{k_0}=\sum_{k_0=k_1}^n\frac{\binom{n}{k_0}}{\binom{n+k_0}{k_0}} (-1)^{k_0}\cdot S^{\sharp}_{k_0, k_1}(\{1\}^{c_1-3})
=\frac{(-1)^{k_1}\cdot k_1}{n^{c_1-2}}\cdot \frac{\binom{n}{k_1}}{\binom{n+k_1}{k_1}},
$$
which implies
\begin{equation*} 
\begin{split}
&\frac{n^2}{n^2-z_0^2}\,F_{n-1}(c_1, \ldots, c_d; z_0, z_1, \ldots, z_d)-\Sigma_0 \\
&\quad= \frac{2}{n^{c_1-2}}\cdot\frac{1}{n^2-z_0^2}\sum_{n\ge k_1\ge\cdots\ge k_d\ge 1}\frac{\binom{n}{k_1}}{\binom{n+k_1}{k_1}}
\frac{(-1)^{k_1}\cdot k_1\cdot(-1)^{k_1\delta(c_2)} k_1^{\delta(c_2)-1}}{k_1^2-z_1^2} \prod_{i=2}^d\frac{(-1)^{k_i\delta_i} k_i^{\delta_i-1}}{k_i^2-z_i^2} \\
&\quad\times S^{\sharp}_{k_{i-1}, k_i}(\{1\}^{c_i-3}) 
=\frac{-n^{2-c_1}}{n^2-z_0^2}\cdot F_n(c_2, \ldots, c_d; z_1, \ldots, z_d),
\end{split}
\end{equation*}
and therefore, the proof is complete.
\end{proof}

\begin{corollary} \label{C1}
For any integers $n\ge 1$, $d\ge 0$, $c_1, \ldots, c_d\in \{1, 3\}$, and any complex numbers $z_0, z_1, \ldots, z_d$ with $|z_j|<1, j=0, 1,\ldots, d$,
we have
\begin{equation*}
F_n(c_1, \ldots, c_d; z_0, z_1, \ldots, z_d)
=-\sum_{n\ge k_0\ge k_1\ge\cdots\ge k_d\ge 1}\frac{\binom{n}{k_0}}{\binom{n+k_0}{k_0}} \prod_{i=0}^d
\frac{(-1)^{k_i\delta_i}k_i^{\delta_i-1}}{k_i^2-z_i^2}\,2^{\Delta(k_{i-1}, k_i)}
\end{equation*}
with the same notation as in Theorem {\rm\ref{T1}}.
\end{corollary}

\begin{corollary} \label{C2}
For any integers $n\ge 1$, $d\ge 0$, $0\le m\le d$, $c_1, \ldots, c_d\in{\mathbb N}\setminus{\{2\}}$, and any complex numbers $z_0, z_1, \ldots, z_d$ 
with $|z_j|<1, j=0, 1,\ldots, d$,
we have
\begin{equation*} 
\begin{split}
\underset{a_m\ge 1}{\sum_{a_0, a_1, \ldots, a_d\ge 0}}&H_n^{\star}(\{2\}^{a_0}, c_1, \{2\}^{a_1},  \ldots, c_d, \{2\}^{a_d}) \,z_0^{2a_0}z_1^{2a_1}\cdots z_d^{2a_d} \\
&=-z_m^2\sum_{n\ge k_0\ge k_1\ge\cdots\ge k_d\ge 1}\frac{\binom{n}{k_0}}{\binom{n+k_0}{k_0}}\, \frac{1}{k_m^2}\,\prod_{i=0}^d
\frac{(-1)^{k_i\delta_i}k_i^{\delta_i-1}}{k_i^2-z_i^2}\,S^{\sharp}_{k_{i-1}, k_i}(\{1\}^{c_i-3})
\end{split}
\end{equation*}
with the same notation as in Theorem {\rm \ref{T1}}.
\end{corollary}
\begin{proof}
The required formula easily follows from Theorem \ref{T1} and the relation
\begin{equation} \label{ge1}
\begin{split}
&\underset{a_m\ge 1}{\sum_{a_0, a_1, \ldots, a_d\ge 0}}H_n^{\star}(\{2\}^{a_0}, c_1, \{2\}^{a_1},  \ldots, c_d, \{2\}^{a_d}) \,z_0^{2a_0}z_1^{2a_1}\cdots z_d^{2a_d} \\
&\qquad\qquad=F_n(c_1, \ldots, c_d; z_0, z_1, \ldots, z_d)-F_n(c_1, \ldots, c_d; z_0, z_1, \ldots, z_{m-1}, 0, z_{m+1}, \ldots, z_d).
\end{split}
\end{equation}
\end{proof}

\begin{corollary} \label{C3}
For any integers $n\ge 1$, $d\ge 0$,  $c_1, \ldots, c_{d}\in{\mathbb N}\setminus{\{2\}}$, and any complex numbers $z_0,  \ldots, z_{d}$ 
with $|z_j|<1, j=0, 1,\ldots, d$,
we have
\begin{equation} \label{eq06} 
\begin{split}
\sum_{a_0, a_1, \ldots, a_{d}\ge 0}&H_n^{\star}(\{2\}^{a_0}, c_1, \{2\}^{a_1},  \ldots, c_{d}, \{2\}^{a_{d}}, 1) \,z_0^{2a_0}z_1^{2a_1}\cdots z_{d}^{2a_{d}} \\
&=\sum_{n\ge k_0\ge k_1\ge\cdots\ge k_{d}\ge 1}\frac{\binom{n}{k_0}}{\binom{n+k_0}{k_0}} \prod_{i=0}^{d}
\frac{(-1)^{k_i\delta_i}k_i^{\delta_i-1}}{k_i^2-z_i^2}\,S^{\sharp}_{k_{i-1}, k_i}(\{1\}^{c_i-3}),
\end{split}
\end{equation}
where $c_0=c_{d+1}=1$, $k_{-1}=0$, $\delta_i=\delta(c_i)+\delta(c_{i+1})$, and $\delta(c)$
is defined  in Theorem {\rm \ref{T1}}.
\end{corollary}
\begin{proof}
Note that the generating function on the left of (\ref{eq06}) is exactly $$F_n(c_1, \ldots, c_{d}, 1; z_0, z_1, \ldots, z_{d}, 0).$$
Therefore, by Theorem \ref{T1} with $d$ replaced by $d+1$, the inner sum over $k_{d+1}$ is reduced to
\begin{equation*}
\sum_{k_{d+1}=1}^{k_{d}}\frac{(-1)^{k_{d+1}}}{k_{d+1}^2-0^2}\cdot k_{d+1}^2\cdot 2^{\Delta(k_{d}, k_{d+1})}
=\sum_{k_{d+1}=1}^{k_{d}}(-1)^{k_{d+1}}\cdot 2^{\Delta(k_{d}, k_{d+1})}=2\sum_{k_{d+1}=1}^{k_{d}-1}(-1)^{k_{d+1}}+(-1)^{k_{d}}=-1,
\end{equation*}
since 
$$
\sum_{k_{d+1}=1}^{k_{d}-1} (-1)^{k_{d+1}}=\begin{cases}
-1, &\quad\text{if \,  $k_{d}$ is even}; \\
0, &\quad\text{if \,  $k_{d}$ is odd},
\end{cases}
$$
and the formula follows.
\end{proof}

\begin{corollary} \label{C3.5}
For any integers $n\ge 1$, $d\ge 0$,  $0\le m\le d$, $c_1, \ldots, c_{d}\in{\mathbb N}\setminus{\{2\}}$, and any complex numbers $z_0,  \ldots, z_{d}$ 
with $|z_j|<1, j=0, 1,\ldots, d$,
we have
\begin{equation*} 
\begin{split}
\underset{a_m\ge 1}{\sum_{a_0, a_1, \ldots, a_{d}\ge 0}}&H_n^{\star}(\{2\}^{a_0}, c_1, \{2\}^{a_1},  \ldots, c_{d}, \{2\}^{a_{d}}, 1) \,z_0^{2a_0}z_1^{2a_1}\cdots z_{d}^{2a_{d}} \\
&=z_m^2\sum_{n\ge k_0\ge k_1\ge\cdots\ge k_{d}\ge 1}\frac{\binom{n}{k_0}}{\binom{n+k_0}{k_0}}\,\frac{1}{k_m^2}\, \prod_{i=0}^{d}
\frac{(-1)^{k_i\delta_i}k_i^{\delta_i-1}}{k_i^2-z_i^2}\,S^{\sharp}_{k_{i-1}, k_i}(\{1\}^{c_i-3}),
\end{split}
\end{equation*}
where $c_0=c_{d+1}=1$, $k_{-1}=0$, $\delta_i=\delta(c_i)+\delta(c_{i+1})$, and $\delta(c)$
is defined in Theorem {\rm \ref{T1}}.
\end{corollary}
\begin{proof}
The formula follows from identity (\ref{ge1}) with  $d$ replaced by $d+1$ and then setting
$c_{d+1}=1$, $z_{d+1}=0$, and applying Corollary~\ref{C3}.
\end{proof}

\begin{corollary} \label{C4}
For any integers $n\ge 1$, $d\ge 0$,  $c_1, \ldots, c_{d}\in{\mathbb N}\setminus{\{2\}}$, and  $a_0, a_1  \ldots, a_{d}\ge 0$, 
we have
\begin{equation*} 
H_n^{\star}(\{2\}^{a_0}, c_1, \{2\}^{a_1},  \ldots, c_{d}, \{2\}^{a_{d}})
=-\sum_{n\ge k_0\ge k_1\ge\cdots\ge k_{d}\ge 1}\frac{\binom{n}{k_0}}{\binom{n+k_0}{k_0}}\, \prod_{i=0}^{d}\,
\frac{(-1)^{k_i\delta_i}}{k_i^{2a_i+3-\delta_i}}\,S^{\sharp}_{k_{i-1}, k_i}(\{1\}^{c_i-3})
\end{equation*}
with the same notation as in Theorem {\rm \ref{T1}}, and
\begin{equation*} 
H_n^{\star}(\{2\}^{a_0}, c_1, \{2\}^{a_1},  \ldots, c_{d}, \{2\}^{a_{d}}, 1)
=\!\sum_{n\ge k_0\ge\cdots\ge k_{d}\ge 1}\frac{\binom{n}{k_0}}{\binom{n+k_0}{k_0}}\, \prod_{i=0}^{d}\,
\frac{(-1)^{k_i\delta_i}}{k_i^{2a_i+3-\delta_i}}\,S^{\sharp}_{k_{i-1}, k_i}(\{1\}^{c_i-3}),
\end{equation*}
where $c_0=c_{d+1}=1$, $k_{-1}=0$, $\delta_i=\delta(c_i)+\delta(c_{i+1})$, and $\delta(c)$
is defined in Theorem {\rm \ref{T1}}.
Moreover, if $c_i\in\{1, 3\}$, then $S^{\sharp}_{k_{i-1}, k_i}(\{1\}^{c_i-3})$ is replaced by $2^{\Delta(k_{i-1}, k_i)}$ in the above formulas.
\end{corollary}
\begin{proof}
Expanding $\frac{1}{k_i^2-z_i^2}$ in powers of $z_i$ and comparing coefficients of $z_0^{2a_0}z_1^{2a_1}\cdots z_d^{2a_d}$
on both sides of equation (\ref{T1part1}), we get the first formula. The second formula follows similarly from~(\ref{eq06}).
\end{proof}

\section{Limit transfer to multiple zeta star values}

The purpose of this section is to justify the possibility of limit transfer from generating functions of multiple harmonic star sums
to generating functions of multiple zeta star values.

Let 
$$
F(c_1, \ldots, c_d; z_0, z_1, \ldots, z_d)=\sum_{a_0, a_1, \ldots, a_d\ge 0}\zeta^\star(\{2\}^{a_0}, c_1, \{2\}^{a_1}, \ldots, c_d, \{2\}^{a_d})\,z_0^{2a_0}z_1^{2a_1}\cdots z_d^{2a_d}.
$$
\begin{lemma} \label{transfer}
Let $d\in{\mathbb N}_0$, $z_0, z_1, \ldots, z_d\in{\mathbb C}$, $|z_j|<1$, $j=0,1,\ldots, d$, $c_1, \ldots, c_d\in{\mathbb N}\setminus\{2\}$, and $c_1\ge 3$ if $d\ge 1$. Then
$$
\lim_{n\to\infty} F_n(c_1, \ldots, c_d; z_0, z_1, \ldots, z_d) = F(c_1, \ldots, c_d; z_0, z_1, \ldots, z_d).
$$
Moreover, the convergence is uniform in any closed region $D: |z_0|\le q_0<1$, $|z_1|\le q_1<1, \ldots$, $|z_d|\le q_d<1$.
\end{lemma}
\begin{proof}
For positive integers $n, m, s_1, \ldots, s_r$ and $n\ge m$, let 
$$
H_{n,m}^\star(s_1, \ldots, s_r)=\sum_{n\ge k_1\ge \cdots\ge k_r\ge m} \frac{1}{k_1^{s_1}\cdots k_r^{s_r}}.
$$
Notice that we allow also the case $n=\infty$  in the above definition if $s_1>1$. Then observing that 
$$
\prod\limits_{k=m}^n \left(1-\frac{z^2}{k^2}\right)^{-1} = \sum_{l=0}^{\infty}z^{2l} \cdot H_{n,m}^\star(\{2\}^l), \qquad |z|<1,
$$
we have the following representations:
\begin{equation} \label{eq11}
\begin{split}
F_n(&c_1, \ldots, c_d; z_0, z_1, \ldots, z_d) \\
&=\sum_{n\ge k_1\ge k_2\ge\cdots\ge k_d\ge 1}\frac{1}{\prod\limits_{k=k_1}^n\left(1-\frac{z_0^2}{k^2}\right)
\cdot k_1^{c_1}\cdot \prod\limits_{k=k_2}^{k_1}\left(1-\frac{z_1^2}{k^2}\right)\cdot k_2^{c_2} \cdots
k_d^{c_d}\cdot \prod\limits_{k=1}^{k_d}\left(1-\frac{z_d^2}{k^2}\right)},
\end{split}
\end{equation}
if $d\ge 1$, and
\begin{equation} \label{D}
F_n(\,; z_0)=\prod\limits_{k=1}^n\left(1-\frac{z_0^2}{k^2}\right)^{-1}.
\end{equation}
Note that the case $n=\infty$ in (\ref{eq11}) and (\ref{D}) corresponds to $F(c_1, \ldots, c_d; z_0, z_1, \ldots, z_d)$, i.e., 
\begin{equation} \label{inf}
F_{\infty}(c_1, \ldots, c_d; z_0, z_1, \ldots, z_d) = F(c_1, \ldots, c_d; z_0, z_1, \ldots, z_d).
\end{equation}
Then by (\ref{D}) and (\ref{inf}), we easily conclude the theorem for $d=0$.

\noindent Now consider the case $d\ge 1$. From (\ref{eq11}) we have the following  upper bound:
\begin{equation} \label{eq12}
\begin{split}
|F_n(&c_1, \ldots, c_d; z_0, z_1, \ldots, z_d)| \\
&\le\sum_{n\ge k_1\ge k_2\ge\cdots\ge k_d\ge 1}\frac{1}{\prod\limits_{k=k_1}^n\left(1-\frac{|z_0|^2}{k^2}\right)
\cdot k_1^{c_1}\cdot \prod\limits_{k=k_2}^{k_1}\left(1-\frac{|z_1|^2}{k^2}\right)\cdot k_2^{c_2} \cdots
k_d^{c_d}\cdot \prod\limits_{k=1}^{k_d}\left(1-\frac{|z_d|^2}{k^2}\right)} \\
&< \frac{\pi|z_0|}{\sin(\pi|z_0|)}\cdot\frac{\pi|z_1|}{\sin(\pi|z_1|)} \cdots \frac{\pi|z_d|}{\sin(\pi|z_d|)}\cdot H_n^\star(c_1, \ldots, c_d) 
< \prod\limits_{j=0}^d\frac{\pi q_j}{\sin(\pi q_j)} \cdot \zeta^\star(c_1, \ldots, c_d)
\end{split}
\end{equation}
for all $z_0, z_1, \ldots, z_d\in{\mathbb C}$ such that $|z_0|\le q_0<1$, $|z_1|\le q_1<1,\ldots$,  $|z_d|\le q_d<1$.

\noindent Let 
\begin{equation*} 
\begin{split}
F_{n, \infty}(&c_1, \ldots, c_d; z_0, z_1, \ldots, z_d) \\
&=\sum_{n\ge k_1\ge k_2\ge\cdots\ge k_d\ge 1}\frac{1}{\prod\limits_{k=k_1}^{\infty}\left(1-\frac{z_0^2}{k^2}\right)
\cdot k_1^{c_1}\cdot \prod\limits_{k=k_2}^{k_1}\left(1-\frac{z_1^2}{k^2}\right)\cdot k_2^{c_2} \cdots
k_d^{c_d}\cdot \prod\limits_{k=1}^{k_d}\left(1-\frac{z_d^2}{k^2}\right)}.
\end{split}
\end{equation*}
Then
\begin{equation}  \label{13}
\begin{split}
|F_{n, \infty}&(c_1, \ldots, c_d; z_0, z_1, \ldots, z_d) -  F_{n}(c_1, \ldots, c_d; z_0, z_1, \ldots, z_d)|\\
&\le\sum_{n\ge k_1\ge\cdots\ge k_d\ge 1}\frac{1}{\prod\limits_{k=k_1}^{n}\left(1-\frac{|z_0|^2}{k^2}\right)
\cdot k_1^{c_1} \cdots
k_d^{c_d}\cdot \prod\limits_{k=1}^{k_d}\left(1-\frac{|z_d|^2}{k^2}\right)}\cdot
\Biggl|\frac{1}{\prod\limits_{k=n+1}^{\infty}\left(1-\frac{z_0^2}{k^2}\right)}-1\Biggr|.
\end{split}
\end{equation}
Since
\begin{equation*}
\Biggl|\frac{1}{\prod\limits_{k=n+1}^{\infty}\left(1-\frac{z_0^2}{k^2}\right)}-1\Biggr|=\left|\sum_{l=1}^{\infty}H_{\infty, n+1}^\star(\{2\}^l) z_0^{2l}\right|
\le \sum_{l=1}^{\infty}H_{\infty, n+1}^\star(\{2\}^l) |z_0|^{2l} < \sum_{l=1}^{\infty}H_{\infty, n+1}^\star(\{2\}^l)
\end{equation*}
and
\begin{equation*}
H_{\infty, n+1}^{\star}(\{2\}^l)=\sum_{k_1\ge \cdots \ge k_l\ge n+1}\frac{1}{k_1^2\cdots k_l^2} < \left(\sum_{k=n+1}^{\infty}\frac{1}{k^2}\right)^l
< \left(\int_n^{\infty}\frac{dx}{x^2}\right)^l = \frac{1}{n^l}
\end{equation*}
we get
\begin{equation} \label{eq14}
\Biggl|\frac{1}{\prod\limits_{k=n+1}^{\infty}\left(1-\frac{z_0^2}{k^2}\right)}-1\Biggr| < \sum_{l=1}^{\infty}\frac{1}{n^l}=\frac{1/n}{1-1/n}=\frac{1}{n-1}
\end{equation}
and therefore, by (\ref{eq12})--(\ref{eq14}),
\begin{equation} \label{eq15}
|F_{n, \infty}(c_1, \ldots, c_d; z_0, \ldots, z_d) -  F_{n}(c_1, \ldots, c_d; z_0, \ldots, z_d)| < \frac{1}{n - 1} 
 \prod\limits_{j=0}^d\frac{\pi q_j}{\sin(\pi q_j)} \cdot \zeta^\star(c_1, \ldots, c_d) \to 0 
\end{equation}
as $n\to \infty$  for all $z_0, z_1, \ldots, z_d\in{\mathbb C}$ such that $|z_0|\le q_0<1$, $|z_1|\le q_1<1,\ldots$,  $|z_d|\le q_d<1$.

\noindent From the other side, 
\begin{equation} \label{eq16}
\begin{split}
 |F(c_1, &\ldots, c_d; z_0, \ldots, z_d) -  F_{n, \infty}(c_1, \ldots, c_d; z_0, \ldots, z_d)|  \\
 &\le \underset{k_1>n}{\sum_{k_1\ge k_2\ge\cdots\ge k_d\ge 1}}\frac{1}{\prod\limits_{k=k_1}^{\infty}\left(1-\frac{|z_0|^2}{k^2}\right)
\cdot k_1^{c_1}\cdot \prod\limits_{k=k_2}^{k_1}\left(1-\frac{|z_1|^2}{k^2}\right)\cdot k_2^{c_2} \cdots
k_d^{c_d}\cdot \prod\limits_{k=1}^{k_d}\left(1-\frac{|z_d|^2}{k^2}\right)} \\
& <  \prod\limits_{j=0}^d\frac{\pi q_j}{\sin(\pi q_j)} \cdot \Bigl(\zeta^\star(c_1, \ldots, c_d) - H_n^\star(c_1, \ldots, c_d)\Bigr) \to 0 \quad (n\to \infty)
 \end{split}
 \end{equation}
uniformly on $D$. Now combining (\ref{eq15}) and (\ref{eq16}), we conclude the proof.
\end{proof}

Taking the limit as $n\to\infty$ in Theorem \ref{T1}  by Lemmas \ref{L3} and \ref{transfer}, we get 
Theorem \ref{T2}.  Theorem \ref{CZ4} follows from Corollary \ref{C4}
by taking the limit $n\to\infty$ and applying Lemma~\ref{L3}.  Theorem \ref{CZ2} for $m\ge 1$ follows from (\ref{ge1}), Lemma \ref{transfer},
and Lemma \ref{L3} by taking the limit $n\to\infty$.
Theorem \ref{CZ2}  for $m=0$ follows from Theorem \ref{CZ4}
by summing identity (\ref{CZ4eq1}) over the corresponding set of integers $a_0, a_1, \ldots, a_d$. 
Corollary \ref{CZ3.5} follows from Theorem \ref{CZ2} by setting $c_d=1$, $z_d=0$, and applying the same argument as in the proof of Corollary \ref{C3}.
Theorems~\ref{2-3}--\ref{2-3-1} are simple consequences of Theorem \ref{T2} and Theorem \ref{CZ2}.
Note also that Theorem \ref{T2} can be obtained from Theorem \ref{CZ4} by summation over $a_0, a_1, \ldots, a_d\ge 0$.

\section*{Acknowledgement}

The authors would like to thank the anonymous referee for the careful reading of the manuscript and helpful comments and suggestions.

\end{document}